\documentclass[10pt]{amsart}
\usepackage{amsmath}
\usepackage{amsfonts}
\usepackage{amssymb}
\usepackage{latexsym}
\usepackage{amsthm}
\makeatletter
\@namedef{subjclassname@2010}{%
  \textup{2010} Mathematics Subject Classification}
\makeatother
\usepackage{setspace}

\usepackage{hyperref}
\usepackage{xcolor}

\newtheorem*{hypothesis a}{Hypothesis A}

\everymath{\displaystyle}

\newtheorem{thm}{Theorem}[section]

\newtheorem{conj}[thm]{Conjecture}

\newtheorem{example}[thm]{Example}

\theoremstyle{definition}
\newtheorem{df}[thm]{Definition}

\newtheorem{rmk}[thm]{Remark}

\numberwithin{equation}{section}

\newcommand{\Norm}{\mbox{\normalfont \bf N}}

\newcommand{\rb}{\overline\rho}
\newcommand{\Q}{\mathbb{Q}}
\newcommand{\C}{\mathbb{C}}
\newcommand{\Z}{\mathbb{Z}}

\newcommand{\e}{\epsilon}

\newcommand{\fm}[4]{\begin{pmatrix} #1 & #2 \\ #3 & #4 \end{pmatrix}}

\begin{document}
\renewcommand*\arraystretch{2}
 \title{A characterization of ordinary modular eigenforms with CM}
\author{Rajender Adibhatla}
\address{ Department of Mathematics, University of Luxembourg,
                          Campus Kirchberg, 6 rue Richard Coudenhove-Kalergi, L-1359 Luxembourg }
                     \email{rajender.adibhatla@uni.lu}

\author{Panagiotis Tsaknias}
\address{Department of Mathematics, University of Luxembourg,
                          Campus Kirchberg, 6 rue Richard Coudenhove-Kalergi, L-1359 Luxembourg}
                          \email{panagiotis.tsaknias@uni.lu}

\date{July 2013}

\begin{abstract}
 For a rational prime $p \geq 3$ we show that a $p$-ordinary modular eigenform $f$ of weight $k\geq 2$, with $p$-adic Galois
representation $\rho_f$, mod ${p^m}$ reductions
$\rho_{f,m}$,  and with complex multiplication (CM), is characterized by the existence of $p$-ordinary CM companion forms $h_m$ modulo $p^m$  for all integers $m \geq 1$ in the sense that $\rho_{f,m}\sim \rho_{h_m,m}\otimes\chi^{k-1}$,  where $\chi$ is the $p$-adic cyclotomic character.  \end{abstract}

\maketitle

\section{Introduction}
For a rational prime $p \geq 3$ let $f$ be a primitive, cusp form with $q$-expansion $\sum a_n(f)q^n$ and associated $p$-adic Galois representation $\rho_f$. In this paper, we prove some interesting arithmetic properties of such a form which, in addition, has complex multiplication (CM). One of the reasons that CM forms have historically been an important subclass of modular forms is  the simplicity with which they can be expressed. They arise from algebraic Hecke characters of imaginary quadratic fields and this makes them ideal initial candidates on which one can  check deep conjectures in the theory of modular forms. 
The specific arithmetic property we establish involves  higher congruence companion forms which were introduced in \cite{Raja1}. The precise definition and properties of these forms are given in Section $2$, but for now we will only remark that companion forms mod $p^m$ are defined as natural analogs of the classical (mod $p$) companion forms of Serre and Gross \cite{Gross}.  

The main theorem of this work (Theorem 4.1, stated and proved in Section $4$) establishes that given a $p$-ordinary CM form $f$, one can always find a CM companion form mod $p^m$ for any integer $m \geq 1$. The proof explicitly finds the desired companion forms in a Hida family of CM forms, thereby circumventing the deformation theory and modularity lifting approach of the companion form theorem in \cite{Raja1}. As an application of the main theorem we show that the converse is true as well: Any $p$-ordinary form $f$ which has CM companions mod $p^m$ for each $m \geq 1$ must necessarily have CM.  We therefore have a complete arithmetical  characterization of $p$-ordinary CM forms. In Section $5$ we present a different proof of the main theorem. It is elementary in the sense that we work directly with the Hecke characters and avoid the heavy machinery of Hida families and overconvergent modular forms. Moreover, the result in this section is, in fact,  a strengthening of the main theorem because it shows the existence of companion forms for odd composite moduli $M$ and not just modulo $p^m$.  On the other hand, the argument appears to work only under the hypothesis  that the class number of the imaginary quadratic extension from which $f$ arises is coprime to $M$.
  
\section{Higher companion forms}
Let $p \geq 3$ be a rational prime and let $f$ be a primitive  cusp form of weight $k \geq 2 $, level $N$ prime to $p$, Nebentypus $\Psi$  and $q$-expansion $\sum a_n(f)q^n$ .  Here, $f$ is primitive in the sense that it is a normalised newform that is a common eigenform for all the Hecke operators. For a place $\mathfrak{p}|p$ in $K_f$, the number field  generated by the $a_n(f)$'s, let $\rho_f: G_\Q \longrightarrow GL_2(K_{f,\mathfrak{p}})$ be a continuous, odd, irreducible Galois representation that can be attached to $f$. We may, after conjugation, assume that $\rho_f$ takes values in the ring of integers of some finite extension of $\Q_p$. This allows us to consider reductions of $\rho_f$ modulo $p^m$ (for integers $m \geq 1$) which, with the exception of the mod $p$ reduction $\rb$, we will denote by $\rho_{f,m}$. 
We can, and will,  define congruences and reduction mod $p^m$ even when the elements don't lie in $\mathbb{Z}_p$ by using the notion of congruence due to Wiese and Ventosa \cite{Wiese-Ventosa}.
With $f$ as above, $p$-adic cyclotomic character $\chi$ and  $2 \leq k \leq p^{m-1}(p-1)+1$, we define  a companion form of $f$ to be as follows.

\begin{df}\label{compdef} Let  $k_m \geq 2$ be the smallest integer such that $k_m+k-2 \equiv 0$ mod $p^{m-1}(p-1)$. Then,  
 a \textbf{companion form} $g$ of $f$, modulo $p^m$, is a  $p$-ordinary normalised eigenform of weight $k_m$ such that $\rho_{f,m} \simeq \rho_{g,m}\otimes\chi^{k-1} \mod{p^m}$. An equivalent formulation of the above criterion in terms of the Fourier expansions is:  $a_n(f) \equiv n^{k-1}a_n(g)$ mod $p^{m-1}(p-1)$ for  $(n,Np) = 1$.
\end{df}

\begin{rmk}
The equivalence between the Galois side and the coefficient side in the above definition perhaps needs further justification. One direction is immediate if we take the traces of Frobenii of the Galois representations. The other direction follows from the absolute irreducibility of  $\rb$, Chebotarev and \cite[Th\'{e}or\`{e}me 1]{Car94} which is essentially a generalization of the Brauer-Nesbitt theorem to arbitrary local rings. \end{rmk}






Following Wiles \cite{W1}, we say that $f$ is \textbf{ordinary} at $p$ (or simply \textbf{$p$-ordinary}) if $a_p(f) \not\equiv 0$  mod $\mathfrak{p}$ for each prime $\mathfrak{p}|p$ (in $K_f$).  Then, by Wiles \cite{W1} and Mazur-Wiles \cite{MW}, we have
 \[
 \rho_f|_{G_p}\sim \fm{\chi^{k-1}\lambda(\Psi(p)/a_p)}{*}{0}{\lambda(a_p)}
 \]
  where
 $G_p$ is a decomposition group at $p$ and, for any $\alpha$ in $(\mathcal{O}_{K_f})^{\times}$, $\lambda(\alpha)$ is the unramified character taking $Frob_p$ to $\alpha$.

If $\rho_{f}$ mod $p$ is absolutely irreducible,  then we note that the reductions $\rho_{f,m}|_{G_p}=\rho_f|_{G_p}$ mod ${p^m}$ for $m \geq 1$ are  independent of the  choice of 
lattice used to define $\rho_f$. A natural question is to ask when  the restriction $\rho_{f,m}|_{G_p}$ actually splits. It is an easy check that  a sufficient condition for  $\rho_{f,m}$ to split at $p$ is that $f$ has a $p$-ordinary companion form modulo $p^m$. 
 
That the  splitting at $p$ of $\rho_{f,m}$ implies the existence of a companion form mod $p^m$ is considerably more difficult to prove. This was shown in \cite{Raja1} but only under the hypothesis that $\rb_f$ has full image. 
However, when $f$ has CM, Im($\rb_f$) is necessarily projectively dihedral.  This renders the method of \textit{loc. cit.} inapplicable even though, \textit{a priori}, we know that $\rho_f|_{G_p}$ splits and therefore expect that $f$ should have a companion mod $p^m$ for all $m \geq 1$.
 In the sequel we avoid the use of lifting theorems altogether and exhibit the companion forms by working directly with the  Hecke character associated to the CM form and visualising them as part of a Hida family. 

\section{Hecke characters and CM forms}

In this section we briefly describe the well-known connection between forms with complex multiplication (CM forms) and Hecke characters of imaginary quadratic fields.  The interested reader may consult \cite[\S 3]{Rib} and \cite[Chapter VII, \S 3]{Neukirch}.

Let $K$ be a number field, $\mathcal{O}=\mathcal{O}_K$ its ring of integers and $\mathfrak{m}$ an ideal of $K$. We denote by $J^\mathfrak{m}$ the group of fractional ideals of $\mathcal{O}_K$ that are coprime to $\mathfrak{m}$. 

\begin{df}  A Hecke character $\psi$ of $K$, of modulus $\mathfrak{m}$, is a group homomorphism $\psi:J^\mathfrak{m}\to \C^*$ such that there exists a character $\psi^\infty:(\mathcal{O}_K/\mathfrak{m})^*\to\C^*$ and a group homomorphism $\psi_\infty:K_\mathbb{R}^*\to\C^*$, where $K_\mathbb{R}:=K\otimes_\Q\mathbb{R}$, such that:
\begin{equation}\label{eq:comp}
\psi((\alpha))=\psi^\infty(\alpha)\psi_\infty(\alpha)\qquad\textrm{for all }\alpha\in K.
\end{equation}
\end{df}
We refer to $\psi^\infty$ and $\psi_\infty$ as the finite and the infinity type of $\psi$ respectively. The conductor of $\psi$ is the conductor of $\psi^\infty$ and we will call $\psi$ primitive if its conductor and modulus are equal.
We now specialize to the imaginary quadratic field case. Let $K=\Q(\sqrt{-d})$, where $d$ is a squarefree positive integer. We will denote by $D$ its discriminant and by $(D|.)$  the Kronecker symbol associated to it. The only possible group homomorphisms $\psi_\infty$ are of the form $\sigma^u$, where $\sigma$ is one of the two conjugate complex embeddings of $K$ and $u \geq 0$ is an integer. The following well known theorem \cite[Theorem 4.8.2]{Miyake} due to Hecke and Shimura then associates a modular form to the Hecke character $\psi$ of $K$.
\begin{thm}\label{thm:Miyake} 
Given a Hecke character $\psi$ of infinity type $\sigma^u$ and  finite type $\psi^{\infty}$ with modulus $\mathfrak{m}$,  assume $u>0$ and let $\Norm\mathfrak{m}$ be the norm of $\mathfrak{m}$. Then, 
$f=\sum_n(\sum_{\mathfrak{\Norm a}=n}\psi(\mathfrak{a}))q^n$
is a cuspidal eigenform in $S_{u+1}(N,\epsilon)$, where $N=|D|\Norm\mathfrak{m}$ and $\epsilon(m)=(D|m)\psi^{\infty}(m)$ for all integers $m$.
\end{thm}

The eigenform associated with a Hecke character $\psi$ is new if and only if $\psi$ is primitive. (See Remark 3.5 in \cite{Rib}.) An important feature of forms which arise in this way is that they coincide with CM forms whose definition we recall.

\begin{df} A  newform $f$   is said
 to have \textbf{complex multiplication}, or just CM, by a
quadratic character $\phi:G_\Q\longrightarrow \{\pm 1\}$ if
$a_q(f)=\phi(q)a_q(f)$  for almost all primes
$q$.
We  will also refer to CM by the
corresponding quadratic extension.
\end{df}

It is clear that the cusp form $f$ in the above theorem has CM by $(D|.)$. Indeed, the coefficient  $a_q(f)$ in the Fourier expansion of $f$ is $0$ if no ideal of $K$ has norm equal to $q$. Since $(D|q) =-1$ exactly when this holds for $q$, $a_q(f) =(D|q)a_q(f)$, and $f$ has CM. On the other hand, Ribet \cite{Rib} shows that if $f$ has CM by an imaginary quadratic field $K$ then it is induced from a Hecke character on $K$.

\section{The main theorem}
We first describe the setting in which we prove our main result.
Let $M_k^\dagger(N)$ denote the space of overconvergent modular forms of level $N$ and weight $k$ in the sense of Coleman \cite{Coleman} and let $M^\dagger(N)$  be the graded ring of overconvergent modular forms of level $N$. One defines an operator
$$\theta:=q\frac{d}{dq}:M^\dagger(N)\to M^\dagger(N), \sum a_nq^n\mapsto \sum na_nq^n.$$
Note in particular that if $f$ has (finite) slope $\alpha$ then $\theta(f)$ has slope $\alpha+1$. Furthermore, by \cite[Proposition 4.3]{Coleman}, 
$$\theta^{k-1}:M_{2-k}^\dagger(N)\to M_k^\dagger(N),$$
and classical CM forms of slope $k-1$ lie in this image (\cite[Proposition 7.1]{Coleman}). We exploit this property of CM forms to prove our main theorem which is the following.

\begin{thm}\label{thm:straight}
Let  $p \geq$ be a rational prime and $f = \sum a_n(f)q^n$  be a $p$-ordinary  primitive CM cuspidal eigenform of weight $k \geq 2$ and level $N$ prime to $p$. Then for every integer $m\geq 1$, there exists a $p$-ordinary CM eigenform $h_m = \sum a_n(h_m)q^n$ of weight  $k_m$ that is  a companion form for $f$ mod $p^m$. \end{thm}

\begin{proof}
Let $K=\Q(\sqrt{-d})$ be the imaginary quadratic field by which $f$ has CM, $D$ be the discriminant of $K$ and $\sigma:K\hookrightarrow \C$ one of its two complex embeddings which we fix for the rest of this section.
Let $g$ be the ($p$-old) eigenform of level $Np$ and weight $k$, whose $p$-th coefficient has $p$-adic valuation equal to $k-1$ and whose $q$-expansion agrees with that of $f$ at all primes $n$ coprime to $p$. Clearly $g$ has CM by $K$ as well. Let $\psi$ be the Hecke character over $K$ of conductor $\mathfrak{m}$ and infinity type $\sigma^{k-1}$ associated with the CM form $g$. The $p$-ordinariness of $f$ implies that $p=\mathfrak{p}\bar{\mathfrak{p}}$ is split in $K$. If $r$ is a rational prime coprime to $\mathfrak{m}$ that splits in $K$, say $r=\mathfrak{r}\bar{\mathfrak{r}}$, then we have:

$$a_r(g)=\psi(\mathfrak{r})+\psi({\bar{\mathfrak{r}}})$$

In the proof of Proposition 7.1 in \cite{Coleman} it is shown how to obtain a $p$-ordinary (CM) $p$-adic eigenform $h_{2-k}$ of weight $2-k$ such that
$$\theta^{k-1}(h_{2-k})=g.$$
 It can be easily seen that $\theta^{k-1}$ has the following effect on $q$-expansions:
$$\theta^{k-1}(\sum c_nq^n)=\sum n^{k-1}c_nq^n.$$
Therefore,
\begin{equation}
a_n(f)=n^{k-1}a_n(h_{2-k}).
\end{equation}
We will denote by $L$ the extension of $\Q_p$ generated by the coefficients of $\psi$ and $\lambda$ and by $\mathfrak{P}$ its prime ideal above $p$. We will also denote by $E$ the extension of $\Q_p$ in which $\lambda$ takes its values. 
It is clear from (4.1) that it is enough to find classical forms $h_{k_m}$ of weight $k_m$ that are congruent to $h_{2-k}\mod \mathfrak{P}^m$, where $\mathfrak{P}$ is the ideal above $p$ in $L$.
By Proposition 7.1 in \cite{Coleman} 

$$h_{2-k}=\sum_{\mathfrak{a}}\bar{\psi}^{-1}(\mathfrak{a})q^{\textbf{N}\mathfrak{a}}$$

where the sum runs over all the integral ideals of $\mathfrak{a}$ $K$ away from $\mathfrak{m}$. Ghate in \cite[pp 234-236]{Ghate}, following Hida, shows how to construct a $p$-adic CM family admitting a specific CM form as a specialization. We outline the construction.
Let $\lambda$ be  a Hecke character of conductor $\mathfrak{p}$ and infinity type $\sigma$. We have that $\mathcal{O}_E^\times\cong\mu_E\times W_E$, where $W_E$ is  the pro-$p$ part of $\mathcal{O}_E^\times$. Let $\langle\rangle$ denote the projection from $\mathcal{O}_E^\times$ to $W_E$. One then one gets (part of) the family mentioned above by
$$G(w):=\sum_{\mathfrak{a}}\bar{\psi}^{-1}(\mathfrak{a})\langle\lambda(\mathfrak{a})\rangle^{w-(2-k)}q^{\textbf{N}\mathfrak{a}}.$$
For any integer $w\geq 2$, $\psi_w(\mathfrak{a})=\bar{\psi}^{-1}(\mathfrak{a})\langle\lambda(\mathfrak{a})\rangle^{w-(2-k)}$ defines a Hecke character of infinity type $w-1$, so that by Theorem 3.1, $G(w)$ is a $p$-adic CM eigenform of weight $w$. Moreover all of them are $p$-ordinary (\cite[pp 236]{Ghate}) and therefore classical for weight $w\geq 2$ (\cite[Theorem I]{Hida}). Clearly $G(2-k)=h_{2-k}$. Notice also that all the $\psi_w$ have coefficients in $L$. 
Let $k_E,k_L$ be the residue fields of $E$ and $L$ respectively and consider the composition $\mathcal{O}_E^\times\to k_E^\times \to k_L^\times$, where the first map is the obvious surjection and the second one is the obvious injection. The image has prime-to-$p$ order so the kernel contains $W_E$. In particular $\langle\lambda\rangle\equiv 1\mod \mathfrak{P}$. This implies:
$$\langle\lambda(\mathfrak{a})\rangle^{w-w'}\equiv 1 \mod\mathfrak{P}\qquad\textrm{for all }w,w'\in\Z_p.$$
It then follows easily that if $w\equiv w'\mod p^{m-1}(p-1)$ then,
$$\langle\lambda(\mathfrak{a})\rangle^{w-w'}\equiv 1 \mod\mathfrak{P}^{m'},$$
where $m' = m(e-1) +1$, with $e$ being the ramification degree of $L$ over $\Q_p$.
Consider the members of the family with weight  $k_m \geq 2$ which is the smallest integer such that 
$k+k_m\equiv2\mod p^{m-1}(p-1)$. The previous identity then gives:
\begin{equation}\label{eq:cong}
\langle\lambda(\mathfrak{a})\rangle^{k_m-(2-k)}\equiv 1 \mod\mathfrak{P}^{m'}.
\end{equation}
Since $G(w)=\sum_n\Big(\sum_{\mathfrak{a}=n}\bar{\psi}^{-1}(\mathfrak{a})\langle\lambda(\mathfrak{a})\rangle^{w-(2-k)}\Big)q^n$, 
the $r$-th coefficient of $G(w)$ (for $r$ a rational prime  is):
$$a_r(G(w))=\bar{\psi}^{-1}(\mathfrak{r})\langle\lambda(\mathfrak{r})\rangle^{w-(2-k)}+\bar{\psi}^{-1}(\bar{\mathfrak{r}})\langle\lambda(\bar{\mathfrak{r}})\rangle^{w-(2-k)}.$$
The identity $\mathfrak{r}\bar{\mathfrak{r}}=r$ along with  \eqref{eq:cong} then shows that 
$$a_q(G(k_m))\equiv a_q(h_{2-k})\mod\mathfrak{P}^{m'}.$$
For the primes $r$ that are inert in $K$ the above equivalence is trivially true since in this case $a_r(G(k_m))=0=a_r(h_{2-k})$.
We thus get that for all primes $r$ away from $Np$ the following holds:
$$a_r(G(k_m))\equiv a_r(h_{2-k})\mod\mathfrak{P}^{m'}.$$
As we mentioned before, all the members of $G$ with weight $w\geq2$ are classical forms so every $h_{m}:=G(k_m)$ is classical. Finally the last identity implies that $h_{m}$ is congruent to $h_{2-k}$ modulo $\mathfrak{P}^{m'}$ almost everywhere, as required. \end{proof}
Note that the Fourier coefficient version of Definition 2.1 was used to show companionship in the above proof. As noted in Remark 4.2 following the definition, this formulation can be reconciled with the Galois representation formulation if one knows that $\rb_f$ is absolutely irreducible. For $f$ as in the theorem above, $\rb_f$ is irreducible because it has projectively dihedral image. Absolute irreducibility then follows because $\rb_f$ is odd and $p \geq 3$.

As an immediate application, we use  Theorem 4.1 to in fact prove its converse; therefore giving a complete arithmetic characterization of $p$-ordinary CM forms.
\begin{thm}\label{thm:main-converse}
If  $f$ is a $p$-ordinary cuspidal eigenform such that for every $m\geq 1$ there exists a CM cuspidal eigenform $h_m$ which is a companion of $f$ modulo $p^m$, then $f$ has CM.
\end{thm}

\begin{proof}
Assume that $f$ has CM companions $h_m$ for all $m \geq 1$ and assume that $h_m$ is CM with respect to a non-trivial quadratic character $\epsilon_m:(\Z/D_m\Z)^\times\longrightarrow \{\pm1\}\subset\C^\times$. The companionship property between $f$ and each of the $h_m$'s enforces the following compatibility congruences:
$$h_{m_1}\equiv'h_{m_2}\mod p^{m_1}\qquad\textrm{for all }m_2\geq m_1$$
$$\e_{m_1}\equiv'\e_{m_2}\mod p^{m_1}\qquad\textrm{for all }m_2\geq m_1$$
where $\equiv'$ means ``away from $p$''. The second compatibility congruence, combined with the fact that the characters $\e_m$ are valued in $\pm1$ and $p\geq3$, implies that $\e_m=\e$ for all $m\geq1$. In particular $\e$ is also a non-trivial quadratic character with conductor $D = D_m$. Furthermore, by Theorem \ref{thm:straight}  each of the $h_m$'s has companions everywhere and, in particular,  there exist CM forms $f_m$, such that:
$$f_m\equiv'f\mod p^m\qquad\textrm{for all }m.$$
Each of the $f_m$'s has CM \textit{w.r.t} the same character $\e$ . This, combined with the previous congruence implies that:
$$a_\ell(f)=a_\ell(f)\e(\ell)\qquad\textrm{for all }(\ell,Dp)=1.$$
If $p|D$, then it is clear that $f$ has CM by $\e$ as well. If $(D,p)=1$, then let $\e':(\Z/Dp\Z)^\times\longrightarrow\C^\times$ be the quadratic character that is trivial on $(\Z/p\Z)^\times$ and $\e$ on $(\Z/D\Z)^\times$. Clearly, $f$ has CM by $\e'$. \end{proof}



\section{An elementary approach}
In the section we present a proof of Theorem 4.1 which relies purely on  carefully manipulating the Hecke character from which the CM form $f$ arises and then appealing to Theorem 3.2. The weakness of this approach, which we have been unable to overcome, is a condition on the class number of the imaginary quadratic field $K$ from which $f$ arises. Nevertheless, this method allows us to prove that $f$, in fact, has companion forms modulo an odd integer $M \geq 3$  --  not necessarily a prime power. Let  $k'\geq 2$ be the smallest integer  such that $k+ k'\equiv2\mod \phi(M)$, with $\phi$ being the Euler totient function. We say that $f$ has a companion $h$ modulo $M$ of weight $k'$ if, keeping the terminology of Definition 2.1, $a_n(f) \equiv n^{k-1}a_n(h) \mod{M}$ for  $(n,NM) = 1$.

\begin{thm}\label{thm:elementary}
Let  $M \geq 3$ be an odd integer  and $f = \sum a_n(f)q^n$  be a  CM eigenform of weight $k \geq 2$, level $N$ coprime to $M$ and $p$-ordinary for each $p|M$.  Assume that the class number of the imaginary quadratic field $K$ from which  $f$ arises is coprime to $M$. Then, there exists a $p$-ordinary CM eigenform $h = \sum a_n(h)q^n$ of weight  $k'$ that is  a companion form for $f$ mod $M$. \end{thm}
\begin{proof}

Let $K=\Q(\sqrt{-D})$ be the quadratic imaginary field by which $f$ has $CM$ and $M' = \prod_{i=1}^r p_i$ for primes $p_i$ dividing $M$.  Let $g$ be the  eigenform of level $NM'$ and weight $k$, whose $p_i$-th coefficient has $p_i$-adic valuation equal to $k-1$ for each $p_i|M'$ and whose $q$-expansion agrees with that of $f$ at all primes away from $NM'$. Such a $g$ exists and one can construct it iteratively as follows. Let $g_1$ be the $p_1$-twin oldform of weight $k$ and level $Np_1$. Now $g_1$  is still $p_2$-ordinary so that one can associate to it the $p_2$-twin eigenform of weight $k$ and level $Np_1p_2$. Proceeding iteratively in this manner we have the desired eigenform $g =g_r$ of weight $k$, and level $NM'$. Clearly $g$ has CM by $K$ as well. Let $\psi_g:J^{\mathfrak{m}}\to \C^*$ be the Hecke character over $K$ associated with $g$. We will denote by $\psi_g^{\infty}$ its finite type and by $\psi_{g,\infty}$ its infinity type. Let $\sigma:K\to \C$ be the identity embedding of $K$ in $\C$. Then $\psi_{g,\infty}=\sigma^{k-1}$. Theorem \ref{thm:Miyake} implies that $\Norm\mathfrak{m}| N$. In particular $m$ is coprime to $\mathfrak{m}$, the modulus of $\psi_g$.

Consider the Hecke character $\varphi=\bar{\psi_g}^{-1}$. We have that $\varphi^{\infty}=\psi_g^{\infty}$ and $\varphi_\infty=\sigma^{1-k}=\sigma^{(2-k)-1}$. The reason we are interested in this character is the identity
\begin{equation}\label{key_id}s^{k-1}(\varphi(\mathfrak{s})+ \varphi(\bar{\mathfrak{s}})) = \psi_g(\mathfrak{s}) + \psi_g(\bar{\mathfrak{s}}) = a_s(g),\end{equation}
for every rational prime $s$ that splits in $K$ as $\mathfrak{s}\bar{\mathfrak{s}}$. In view of this identity it is enough to find a Hecke character $\psi'$ congruent to $\varphi$ and with infinity type $\sigma^{k'-1}$.

Assume for simplicity that $D\neq 1,3$. It is then easy to see (for instance the discussion in Section 4 of \cite{Tsaknias}) that the compatibility of a finite and an infinity type is guaranteed by the identity $\psi_g^{\infty}(-1)=(-1)^{k-1}$. One can therefore find a Hecke character with finite type equal to $\psi_g^{\infty}$ and infinity type $\sigma^{k'-1}$, for any $k'\equiv 2-k $ mod $2$. This Hecke character is not unique. In fact it turns out (see for example Lemma 4.1 in \cite{Tsaknias}) that there are $h_K$ choices, where $h_K$ is the class number of $K$. We explain this in more detail.  Consider the group $J^{\mathfrak{m}}/P^{\mathfrak{m}}$ where $J^{\mathfrak{m}}$ is as in Section 2 and $P^{\mathfrak{m}}$ is the subgroup of principal fractional ideals coprime to $\mathfrak{m}$. It is a finite abelian group of order $h_K$. Let $\mathfrak{a}_i$ denote a fixed choice of representatives that generate $J^{\mathfrak{m}}/P^{\mathfrak{m}}$, $c_i$ be their respective orders and  $\alpha_i$ be the generator of the ideal $\mathfrak{a}_i^{c_i}$. Then, given a pair $(\psi^{\infty}, \sigma^{k'-1})$ consisting of a finite and an infinity type, any Hecke character $\psi'$ corresponding to that pair is completely determined by the values $\psi'(\mathfrak{a}_i)$. Let  $d_i(\psi')$ be a non-negative integer and $\zeta_{c_i}$ denote a primitive $c_i$-th root of unity. The only possible values for $\psi'(\mathfrak{a}_i)$ are of the form
$$\sqrt[c_i]{\psi^{\infty}(\alpha_i)}\sqrt[c_i]{\alpha_i^{k'-1}}\zeta^{d_i(\psi')}_{c_i}$$
where $\sqrt[c_i]{\psi^{\infty}(\alpha_i)}$ and $\sqrt[c_i]{\alpha_i^{k'-1}}$ are fixed choices of $c_i$-th roots of $\psi^{\infty}(\alpha_i)$ and 
$\alpha_i^{k'-1}$ respectively.

Let $k'$ be the integer that we defined previously. Since $k'$ has the same parity as $2-k$  we can define a Hecke character $\psi'$ of infinity type $\sigma^{k'-1}$ and finite type $\psi_g^{\infty}$. Let $L$ be the extension of $K$ with $\sqrt[c_i]{\psi_g^{\infty}(\alpha_i)}$, $\sqrt[c_i]{\alpha_i^{k'-1}}$ and $\zeta_{c_i}$ adjoined for all $i$ (and for the fixed $k'$). Let $M=\prod_pp^{t_p}$. Furthermore, for every $p|M$, let $\mathfrak{p}$ be a prime in $K$ such that $p=\mathfrak{p}\bar{\mathfrak{p}}$ and pick a prime $\mathfrak{P}$ above 
$\mathfrak{p}$ in $L$. 
Our goal is to show that the $d_i(\psi')$'s can always be chosen so that
$$\varphi(\mathfrak{a}_i)\equiv\psi'(\mathfrak{a}_i)\pmod{\mathfrak{P}^{t'_\mathfrak{P}}}$$
for all $i$, where $t'_\mathfrak{P}=e(\mathfrak{P}/\mathfrak{p})(t_p-1)+1$. First, notice that if $\alpha_i\equiv 0 \pmod{\mathfrak{p}}$ then the desired congruence is immediate. In what follows we can therefore freely assume that $\alpha_i$ is a unit in $\mathcal{O}_K/\mathfrak{p}^{t_p}\subseteq \mathcal{O}_L/\mathfrak{P}^{t'_\mathfrak{P}}$. Moreover,       
$\mathcal{O}_K/\mathfrak{p}^{t_p}$ is isomorphic to $\Z/p^{t_p}\Z$ because of the splitting condition on the primes dividing $M$.
Since the order of $(\Z/p^{t_p}\Z)^*$ is $\phi(p^{t_p})$, we get that $\alpha^{k'-(k-2)}\equiv1\pmod{\mathfrak{p}^{t_p}}$.
This in turn implies that $\sqrt[c_i]{\alpha^{k'-(2-k)}}$ is a $c_i$-th root of unity in $\mathcal{O}_L/\mathfrak{P}^{t'_{\mathfrak{P}}}$. Since $h_K$ is coprime to $M$ this $c_i$-th root of unity will lift to one in $\mathcal{O}_L$ so all one has to do is to choose $d_i(\psi')$ so that $\zeta_{c_i}^{d_i(\varphi)-d_i(\psi')}$ is this root. In conclusion, the CM form $h$ that one may associate to the Hecke character $\psi'$ by Theorem 3.2 is the desired companion of $f$ mod $M$. Equation (5.1) ensures that $h$  is $p$-ordinary for each $p$ dividing $M$. \end{proof}

\begin{example} Let $f$ be the CM newform of weight $3$ and level $8$ with the following Fourier expansion,

\begin{center}
$q - 2q^2 - 2q^3 + 4q^4 + 4q^6 - 8q^8 - 5q^9 + 14q^{11} - 8q^{12} + 16q^{16} + 2q^{17} + 10q^{18}- 34q^{19} - 28q^{22} + 16q^{24} + 25q^{25} +  \cdots $ \end{center}

It is ordinary at $3$, $11$ and $17$. Using MAGMA \cite{Magma} we find companions $h_{19}$, $h_{31}$ and $h_{59}$ modulo $33$, $51$ and $99$ respectively. The indices denote the weights of the companions; each has level 8 and CM by the quadratic Dirichlet character of conductor $8$. Their Fourier expansions are:


 \begin{center} $h_{19} =q - 512q^2 - 3266q^3 + 262144q^4 + 1672192q^6 - 134217728q^8 - 376753733q^9 - 354349618q^{11} - 856162304q^{12} + 68719476736q^{16} + 119842447106q^{17} + 192897911296q^{18} + 
    335013705758q^{19} + 181427004416q^{22} + 438355099648q^{24} + 3814697265625q^{25} +\cdots $ \end{center}
    
\begin{center}  $h_{31} = q - 32768q^2 - 26595314q^3 + 1073741824q^4 +  871475249152q^6  -$
$35184372088832q^8 +     501419594663947q^9 +
 6656187998706302q^{11} - 28556500964212736q^{12} + 1152921504606846976q^{16}
 - 4422784932886529086q^{17}- 16430517277948215296q^{18} - 23964789267887608402q^{19}
- 218109968341608103936q^{22} + 935739423595322933248q^{24} + 
    931322574615478515625q^{25} - \cdots$  \end{center}
    
 \begin{center}  $h_{59} =  q - 536870912q^2 + 57281430144478q^3 + 288230376151711744q^4 - 
    30752733642330195623936q^6 - 154742504910672534362390528q^8 - 
    1428966457849531926967711205q^9 + 2900908653579886134108518505134q^{11} + 
    16510248157050893929760929349632q^{12} + 
    83076749736557242056487941267521536q^{16} + 
    955027058519269179716584293727217282q^{17} + 
    767170525443087764404272509180968960q^{18} - 
    15840463221028561793718151601779174594q^{19} - 
    1557413474476125533674994636820167262208q^{22 }- 
    8863871985422232654486014081904498704384q^{24} + 
    34694469519536141888238489627838134765625q^{25} - \cdots$  \end{center}

\end{example}

\section*{Acknowledgments}
The first author thanks DFG GRK 1692 at Regensburg, the Hausdorff Institute, Bonn and  the AFR Grant Scheme of FNR Luxembourg for postdoctoral support during the course of this work. The second author acknowledges the support of the DFG Priority Program SPP 1489.

\end{document}